\documentclass[11pt]{article}

\usepackage{amssymb}
\makeatletter

\usepackage[english]{babel}
\usepackage{amsmath}
\usepackage{amssymb}
\usepackage{graphicx}
\usepackage{color}

\usepackage{graphicx}
\usepackage{color}
\usepackage{amsmath,amsthm,amsfonts,epsfig,setspace}
\numberwithin{equation}{section}

\newcommand{\ds}{\displaystyle}
\def\nm{\noalign{\medskip}}
\newtheorem{thm}{Theorem}[section]
\newtheorem{rmk}{Remark}[section]
\newtheorem{cor}{Corollary}[section]

\newtheorem{lem}{Lemma}[section]

\newtheorem{prop}{Proposition}[section]
\newtheorem{asump}{Assumption}[section]

 \def\p{\partial}
\def \Vh0{\stackrel{\circ}{V}_h} \def\to{\rightarrow}

\def\l{\label}  \def\f{\frac}  

\def\l|{\left|}
\def\r|{\right|}

\newcommand{\R}{\mathbb{R}}

\newcommand{\lc}
{\mathrel{\raise2pt\hbox{${\mathop<\limits_{\raise1pt\hbox
{\mbox{$\sim$}}}}$}}}

\newcommand{\gc}
{\mathrel{\raise2pt\hbox{${\mathop>\limits_{\raise1pt\hbox{\mbox{$\sim$}}}}$}}}

\newcommand{\ec}
{\mathrel{\raise2pt\hbox{${\mathop=\limits_{\raise1pt\hbox{\mbox{$\sim$}}}}$}}}

\def\be{\begin{equation}} \def\ee{\end{equation}}

\def\bea{\begin{eqnarray}}  \def\eea{\end{eqnarray}}

\def\beas{\begin{eqnarray*}} \def\eeas{\end{eqnarray*}}

\def\bn{\begin{enumerate}} \def\en{\end{enumerate}}

\def\bd{\begin{description}} \def\ed{\end{description}}

\title{Effective medium theory for acoustic waves in bubbly fluids near Minnaert resonant frequency\thanks{\footnotesize The work of Hai Zhang was supported by a startup fund from HKUST.}}
\date{April 28, 2016}

\author{
Habib Ammari\thanks{\footnotesize Department of Mathematics, 
ETH Z\"urich, 
R\"amistrasse 101, CH-8092 Z\"urich, Switzerland (habib.ammari@math.ethz.ch)}   
\,\,\, \,\, Hai Zhang\thanks{\footnotesize 
Department of Mathematics, 
 HKUST,  Clear Water Bay, Kowloon, Hong Kong (haizhang@ust.hk).}}

\begin{document}
\maketitle

\begin{abstract}
We derive an effective medium theory for acoustic wave propagation in bubbly fluid near Minnaert resonant frequency. We start with a multiple scattering formulation of the scattering problem of an incident wave by a large number of identical small bubbles in a homogeneous fluid. Under certain conditions on the configuration of the bubbles, we justify the point interaction approximation and establish an effective medium theory for the bubbly fluid as the number of bubbles tends to infinity. The convergence rate is also derived. 
As a consequence, we show that near and below the Minnaert resonant frequency, the obtained effective media can have a high refractive index, which is the reason for the super-focusing experiment observed in \cite{lanoy}. Moreover, our results indicate that the obtained effective medium can be dissipative above the Minnaert resonant frequency, while at that frequency, effective medium theory does not hold. Our theory sheds light on the mechanism of the extraordinary wave properties of metamaterials, which include bubbly fluid as an example, near resonant frequencies.  


\end{abstract}

\medskip

\bigskip

\noindent {\footnotesize Mathematics Subject Classification
(MSC2000): 35R30, 35C20.}



\section{Introduction} \label{sec-intro}

The study of acoustic wave propagation in bubbly fluid has a long history and is driven by  practical applications. It arose during World War II in an overall effort to exploit underwater sound in submarine warfare \cite{domenico}. Afterward, air bubble curtains were used to prevent damage of submerged explosives \cite{fessenden, prairie}. Nowadays, bubbles are used to super-focusing acoustic waves on a deep sub-wavelength scale \cite{lanoy}, as well as super-resolution in medical ultrasonic imaging \cite{tanter}. 

A distinctive feature of bubble in fluid is the high contrast between the air density inside  and outside of the bubble. This results in a quasi-static acoustic resonance, called the Minnaert resonance. At or near this resonant frequency, the size of bubble can be three order of magnitude smaller than
 the wavelength of the incident wave and the bubble behaves as a very strong monopole scatterer of sound. The resonance
makes the bubbles a good candidate for acoustic subwavelength resonator.  They have
the potential to be the basic building blocks for acoustic meta-materials which include bubbly fluids \cite{leroy1, leroy2, leroy3, leroy4}. We refer to \cite{hai3} for a rigorous mathematical treatment of Minnaert resonance and the derivation of the monopole approximation. In this paper, we aim to investigate the extraordinary properties of bubbly fluid for acoustic waves and to understand the mechanism behind. 

There are many interesting works on the acoustic bubble problems; see, for instance,  \cite{caflish, caflish2, kargl, leroy1, leroy2}. Effective equations for wave propagation in bubbly fluids have been derived in the low frequency regime where the frequency is much smaller than the Minneart resonance frequency \cite{caflish, caflish2,kargl}.  In this paper, however, we are concerned with wave propagation near the resonant regime. Our main approach consists of the use of point interaction approximation, which is valid for bubbles with small volume fraction, and the study of its continuum limit.  We recall that the idea of point interaction approximation goes back to Foldy's paper \cite{Foldy}. It is a natural tool to analyze a variety of interesting problems in the continuum limit. It was applied in \cite{Figari} on the heat conduction in material with many small holes and in \cite{caflish} on sound propagation in bubbly fluid at frequency much smaller than the Minnaert resonant frequency. As pointed out in \cite{Figari}, in almost all papers that followed Foldy's, the point interaction approximation in not treated as an important approximation in itself and averaging is carried out over the configuration of the points locations. It is worth emphasizing that averaging is not necessary. The continuum limit holds for deterministic sequences satisfying certain conditions  subject to some other conditions that hold for most of realizations in the random case. In this regard, our paper can be viewed as a 
realization of this idea.

The main contributions of our paper are the following.  
First, we rigorously justify the point interaction approximation for the acoustic scattering of bubbles in fluid under certain conditions. Second, we analyze its continuum limit near the Minnaert resonant frequency, and propose conditions on the distribution of the bubbles which guarantee the validity of the continuum limit. Third, we derive the convergence rate as the number of bubbles goes to infinity and prove that the convergence is
valid outside of neighborhoods of the bubbles.  
Compared with \cite{caflish}, the case we consider near the Minnaert resonant frequency is much more intriguing. As our main results demonstrate, the obtained effective medium is very dispersive near and below the Minnaert resonant frequency and depends sensitively on the volume fraction of the bubbles. At very low volume fraction, the effective refractive index does not change much. It increases to high refractive index when the volume fraction increases and becomes infinity after reaching certain threshold volume fraction. The effective medium also depends sensitively on the frequency. At resonant frequencies, we note that we cannot treat the bubbly fluid as an effective medium. However, effective medium theory is possible at nearby frequencies with a balance with the volume fraction. This balance controls the interaction of the bubbles and is critical for the validity of the effective medium theory. Moreover, when the frequency is above the Minnaert resonant frequency, effective medium theory is still possible and the effective medium may be dissipative which is characterized by attenuating the wave fields therein, which is very different from the wave behavior below the Minnaert resonant frequency.   

The paper is organized as follows. We start with a multiple scattering formulation of the scattering problem of an incident plane wave by a large number of identical small bubbles in a homogeneous fluid in Section \ref{sec-bubbly media}. We also propose assumptions on the configuration of the bubbles for the relevance of our effective medium theory.  
In Section \ref{sec-point interaction}, we justify the point interaction approximation for the scattering of bubbles. In Section \ref{sec-effective media}, we first show the well-posedness of the point interaction system and deduce its limiting behavior. Then we derive an effective medium theory as the number of bubbles tends to infinity. The convergence rate is established. As a consequence, we show that near and below the Minnaert resonant frequency for an individual bubble, the effective medium can have a high refractive index,
justifying the super-focusing effect observed in  \cite{lanoy}. Super-focusing is to push the diffraction limit by reducing the focal spot size so that waves can be confined to a length scale significantly smaller than the diffraction limit which is the half  wavelength of the wave in the free space. Super-focusing 
is the counterpart of super-resolution. Using time reversal imaging method, super-resolution can be achieved in media which can super-focus waves. Finally, the proof of a technical result in the paper is given in Section \ref{sec-proof}.


\section{Bubbly medium for acoustic waves and assumptions} \label{sec-bubbly media}

Consider the scattering of acoustic waves by $N$ identical bubbles distributed in a homogeneous fluid in $\R^3$. The bubbles are represented by
$$D^N:=\cup_{1\leq j \leq N}D_j^N,$$ where $D_j^N=y_j^N + s B$ for $1\leq j \leq N$ with $y_j^N$ being the location, $s$ being the characteristic  size and $B$ being the normalized 
bubble which is a smooth and simply connected domain with size of order one. We denote by $\rho_b$ and $\kappa_b$ the density and the bulk modulus of the air inside the bubble respectively, which are different from the corresponding  $\rho$ and $\kappa$ in the background medium $\R^3 \backslash D^N$.

We assume that $0< s \ll 1$, $N\gg 1$ and that $\{y_j^N\} \subset \Omega$. Let $u^i$ be the incident wave which we assume to be a plane wave for simplicity. The scattering can be modeled by the following system of equations: 
\be \label{eq-scattering}
\left\{
\begin{array} {ll}
&\nabla \cdot \f{1}{\rho} \nabla  u^N+ \frac{\omega^2}{\kappa} u^N  = 0 \quad \mbox{in } \R^3 \backslash D^N, \\
&\nabla \cdot \f{1}{\rho_b} \nabla  u^N+ \frac{\omega^2}{\kappa_b} u^N  = 0 \quad \mbox{in } D^N, \\
& u^N_{+} -u^N_{-}  =0   \quad \mbox{on } \partial D^N, \\
&  \f{1}{\rho} \f{\p u^N}{\p \nu} \bigg|_{+} - \f{1}{\rho_b} \f{\p u^N}{\p \nu} \bigg|_{-} =0 \quad \mbox{on } \partial D^N,\\
&  u^N- u^{i}  \,\,\,  \mbox{satisfies the Sommerfeld radiation condition}, 
  \end{array}
 \right.
\ee
where $u^N$ is the total field and $\omega$ is the frequency. 

We introduce four auxiliary parameters to facilitate our analysis:
$$
v = \sqrt{\frac{\rho}{\kappa}}, \,\, v_b = \sqrt{\frac{\rho_b}{\kappa_b}}, \,\, k=\omega v, \,\, k_b= \omega v_b.
$$

We also introduce two dimensionless contrast parameters 
$$
\delta = \f{\rho_b}{\rho}, \,\, \tau= \f{k_b}{k}= \f{v_b}{v} =\sqrt{\f{\rho_b \kappa}{\rho \kappa_b}}. 
$$

By choosing proper physical units, we may assume that both the frequency $\omega$ and the wave speed outside the bubbles are of order one. As a result, the wavenumber $k$ outside the bubbles is also of order one. We assume that there is a large contrast between both the density and bulk modulus inside and outside the bubbles.
However, the contrast between the wave speeds are small. Thus, both the wave speed 
and wavenumber $k_b$ inside the bubbles are of order one. 
To sum up, we assume that $\delta \ll 1 $, $\tau= O(1)$. We also assume that the domain of interest $\Omega$ has size of order one. 

We use layer potentials to represent the solution to the scattering problem (\ref{eq-scattering}). Let the single layer potential $\mathcal{S}_{D}^{k}$ associated with a domain $D$ and wavenumber $k$ be defined by
$$
\mathcal{S}_{D}^{k} [\psi](x) =  \int_{\p D} G(x, y, k) \psi(y) d\sigma(y),  \quad x \in  \p {D},
$$
where $$G(x, y, k)= - \f{e^{ik|x-y|}}{4 \pi|x-y|}$$ is the  Green function of the Helmholtz equation in $\R^3$, subject to the Sommerfeld radiation condition. When $k=0$, we use the short notation 
$$
\mathcal{S}_{D}= \mathcal{S}_{D}^{0}.
$$
We also define boundary integral operator $\mathcal{K}_{D}^{k, *}$ by
$$
\mathcal{K}_{D}^{k, *} [\psi](x)  = \int_{\p D } \f{\p G(x, y, k)}{ \p \nu(x)} \psi(y) d\sigma(y) ,  \quad x \in \p D. 
$$

Then the solution $u^N$ can be written as 
\be \label{Helm-solution}
u^N(x) = \left\{
\begin{array}{lr}
u^{in} + \mathcal{S}_{D^N}^{k} [\psi^N], & \quad x \in \R^3 \backslash \overline{D^N},\\
\nm
\mathcal{S}_{D}^{k_b} [\psi_b^N] ,  & \quad x \in {D^N},
\end{array}\right.
\ee
for some surface potentials $\psi, \psi_b \in  L^2(\p D^N)$. Here, we have used the notations
\beas
L^2(\p D^N) &=& L^2(\p D^N_1) \times L^2(\p D^N_2) \times \cdots \times L^2(\p D^N_N), \\ 
\mathcal{S}_{D^N}^{k} [\psi^N]&= &\sum_{1\leq j \leq N} \mathcal{S}_{D^N_j}^{k} [\psi_j^N], \\
\mathcal{S}_{D}^{k_b} [\psi_b^N]&= & \sum_{1\leq j \leq N} \mathcal{S}_{D^N_j}^{k} [\psi_{bj}^N].
\eeas

Using the jump relations for the single layer potentials, it is easy to derive that $\psi$ and $\psi_b$ satisfy the following system of boundary integral equations:
\be  \label{eq-boundary}
\mathcal{A}^N(\omega, \delta)[\Psi^N] =F^N,  
\ee
where
\[
\mathcal{A}^N(\omega, \delta) = 
 \begin{pmatrix}
  \mathcal{S}_{D^N}^{k_b} &  -\mathcal{S}_{D^N}^{k}  \\
  -\f{1}{2}Id+ \mathcal{K}_{D^N}^{k_b, *}& -\delta( \f{1}{2}Id+ \mathcal{K}_{D^N}^{k, *})
\end{pmatrix}, 
\,\, \Psi^N= 
\begin{pmatrix}
\psi_b^N\\
\psi^N
\end{pmatrix}, 
\,\,F^N= 
\begin{pmatrix}
u^{in}\\
\delta \f{\partial u^{in}}{\partial \nu}
\end{pmatrix}|_{\p D^N}.
\]

One can show that the scattering problem (\ref{eq-scattering}) is equivalent to the boundary integral equations (\ref{eq-boundary}).Furthermore, it is well-known that there exists a unique solution to the scattering problem (\ref{eq-scattering}), or equivalently to the system (\ref{eq-boundary}). 

Throughout the paper, we denote $\mathcal{H} = L^2(\p D^N) \times L^2(\p D^N)$ and  $\mathcal{H}_1 = H^1(\p D^N) \times L^2(\p D^N)$, 
and use $(\cdot, \cdot)$ for the inner product in $L^2$ spaces. 
It is clear that $\mathcal{A}^N(\omega, \delta)$ is a bounded linear operator from $\mathcal{H}$ to $\mathcal{H}_1$, i.e., 
$\mathcal{A}^N(\omega, \delta) \in \mathcal{L}(\mathcal{H}, \mathcal{H}_1)$. 
We also use the following convention: let $a_N$ and $b_N$ be two real numbers which may depend on $N$, then
$$
a_N \lesssim b_N
$$
means that $a_N \leq C \cdot b_N$ for some constant $C$ which is independent of $a_N$, $b_N$ and $N$.
%


We are interested in the case when there is a large number of small identical bubbles distributed in a bounded domain and the incident wave has a frequency near the Minnaert resonant frequency for an individual bubble. We recall that for the bubble given by $D_j^N= y_j^N + sB$, its corresponding 
Minnaert resonant frequency $\omega_M$ is
$$
\omega_M= \f{1}{s}\sqrt{\f{Cap(B) \delta}{\tau^2 v^2 Vol(B)}}, 
$$
where $Cap(B):= (\mathcal{S}_{B}^{-1}(\chi_{\p B}), \chi_{\p B})_{L^2(\p B)}$ and $Vol(B)$ are the capacity and volume of $B$, respectively. Here, $\chi_{\p B}$ denotes the characteristic function of $\partial B$. 

Throughout the paper, we assume that the following assumption holds:
\begin{asump} \label{assump0-1}
The frequency $\omega =O(1)$ and is independent of $N$. Moreover, 
\be  \label{formula-w}
1- (\f{\omega_M}{\omega})^2 = \beta_0 s^{\epsilon_1}
\ee
for some fixed $0 < \epsilon_1 < 1$ and constant $\beta_0$.
\end{asump}

There are two cases depending on whether $\omega > \omega_M$ or 
$\omega < \omega_M$. In the former case, we have $\beta_0 >0$, while in the 
latter case we have $\beta_0 <0$. We shall see later on that acoustic wave propagation is quite 
different in these two cases. In fact, the wave field may be dissipative in the former case while highly oscillatory and propagating in the latter case. 
We also assume the following. 
\begin{asump} \label{asump0-4}
The following identity holds
\be  \label{formula-s}
s^{1-\epsilon_1}\cdot N =\Lambda,
\ee
where $\Lambda$ is a constant independent of $N$. Moreover, we will assume that $\Lambda$ is large. 
\end{asump}

Therefore, we have
\be
\delta = \omega^2 s^2 (1-s^{\epsilon_1}) \cdot \f{\tau^2 v^2 Vol(B)}{Cap(B)}.
\ee

We note that we have rescaled the original physical problem by imposing the condition that $\omega_M$ is 
of order one. Consequently, the physical parameters $s$ and $\delta$ associated with the size and contrast of the bubbles both depend on $N$. Equation (\ref{formula-s}) gives the volume fraction while 
Equation (\ref{formula-w}) controls the deviation of frequency from the Minnaert resonant frequency. In the limiting process when $N \to \infty$, we have $s \to 0, \, \delta \to 0$.

We assume that the size of each bubble is much smaller than the typical distance between neighboring bubbles so that we may simplify the system by point scatterer approximation. 
More precisely, we make the following assumption. 

\begin{asump} \label{assump0}
The following conditions hold:
\[
\left\{
\begin{array}{ll}
\min_{i \neq j } |y^N_i -y^N_j|  \geq  r_N ,\\
s  \ll r_N,
  \end{array}
\right.
\]
where 
$
r_N = \eta N^{-\f{1}{3}} 
$
for some constant $\eta$ independent of $N$.  Here, $r_N$ can be viewed as the minimum separation distance between neighboring bubbles. 
\end{asump}

Following \cite{papanicolaou}, we assume that there exists $\tilde{V}\in L^{\infty}(\Omega)$ such that
\be  \label{asump10}
\Theta^N (A) \to \int_{A} \tilde{V}(x) dx , \quad \mbox{as} \,\, N \to \infty,
\ee
for any measurable subset $A \subset \R^3$, where
$\Theta^N (A)$ is defined by
\[
\Theta^N (A)= \f{1}{N}  \times \{  \mbox{number of points $y_j^N$ in $A \subset \R^3$}\}. 
\]

In addition, we assume that the following condition on the regularity of the  ``sampling'' points $\{y_j^N\}$ holds.
\begin{asump} \label{assump2}
There exists $0< \epsilon_0<1$ such that for all $h \geq 2r_N$: 
\bea
\f{1}{N}\sum_{|x-y_j^N| \geq h } \f{1}{|x-y_j^N|^2} & \lesssim & |h|^{-\epsilon_0}, \quad \mbox{uniformly for all }\,\,x \in \Omega ,  \label{assump2-2}\\
\f{1}{N}\sum_{2r_N \leq |x-y_j^N| \leq 3h } \f{1}{|x-y_j^N|} &\lesssim & |h|, \quad \mbox{uniformly for all }\,\,x \in \Omega. \label{assump2-3}
\eea
\end{asump}

\begin{rmk}
Note that one can choose $\epsilon_0$ to be a small number in Assumption \ref{assump2}.
One can show that (\ref{assump2-2}) and (\ref{assump2-3}) are respectively equivalent to the following ones 
\bea
\max_{l} \{ \f{1}{N}\sum_{|y_l^N-y_j^N| \geq h} \f{1}{|y_l^N-y_j^N|^2} \}  & \lesssim & h^{-\epsilon_0};\\
\max_{l} \{ \f{1}{N}\sum_{2r_N \leq |y_l^N-y_j^N| \leq 3h } \f{1}{|y_l^N-y_j^N|} \} &\lesssim & h.
\eea
Indeed, these estimates follow from the fact that for each $x \in \Omega$ there exists a finite number of points $y_{j_1}^N$,  $y_{j_2}^N$, ... $y_{j_L}^N$ in the neighborhood of $x$ with $L$ independent of $N$ such that
$$
\f{1}{|x-y_j^N|^2} \leq \sum_{1\leq i \leq L}\f{1}{|y_{j_i}^N-y_j^N|^2}, \quad 
\f{1}{|x-y_j^N|} \leq \sum_{1\leq i \leq L}\f{1}{|y_{j_i}^N-y_j^N|},
$$
for all $y_j^N$ such that $|x-y_j^N| \geq h $. 
\end{rmk}

Following \cite{ozawa}, we also assume the following.  

\begin{asump} \label{assump1}
For any $f \in C^{0, \alpha}(\Omega)$ with $0 < \alpha \leq 1$,
\be  
\max_{1\leq j \leq N} | \f{1}{N} \sum_{i \neq j} G(y_j^N,y_j^N,k) f(y_j^N) -
\int_{\Omega} G(y_j^N, y,k) \tilde{V}(y) f(y) dy| \lesssim \f{1}{N^{\f{\alpha}{3}}}\|f\|_{C^{0, \alpha}(\Omega)}.  
\ee
\end{asump}

\begin{rmk}
By decomposing $G(x, y, k)$ into the singular part,  
$G(x, y, 0)$,  and a smooth part, one can show that Assumption \ref{assump1} is equivalent to 
\be  \label{asump0-1}
\max_{1\leq j \leq N} | \f{1}{N} \sum_{i \neq j}\f{1}{|y_i^N-y_j^N|} f(y_i^N) -
\int_{\Omega} \f{1}{|y-y_j^N|} \tilde{V}(y) f(y) dy| \lesssim \f{1}{N^{\f{\alpha}{3}}}|f\|_{C^{0, \alpha}}.  
\ee
\end{rmk}

For $1\leq j \leq N$, denote by  
\beas
u^{i,N}_j&=& u^i + \sum_{i \neq j} \mathcal{S}_{D_i^N} ^k [\psi_i^N], \\
u^{s,N}_j &=&  \mathcal{S}_{D_j^N}^k [\psi_j^N].
\eeas
It is clear that $u^{i,N}_j$ is the total incident field which impinges on the bubble $D_j^N$ and $u^{s,N}_j$ is the corresponding scattered field. In the next section, we shall justify the point interaction approximation. For this purpose, we need an additional 
assumption. 

\begin{asump} \label{assump0-3}
$\epsilon_0 < \f{3 \epsilon_1}{ 1- \epsilon_1}$. 
\end{asump}

\begin{rmk}
Assumptions \ref{assump0-1} and \ref{assump0-3} are important in our justification of the point interaction approximation, see Proposition \ref{prop-point-approx}. 
The assumption that $\epsilon_1 >0$ is critical here. For the case $\epsilon_1= 0$, the frequency is away from the Minnaert resonant frequency. The scattering coefficient $g$ has magnitude of order $s$. 
The fluctuation in the scattered field from all the other bubbles may generate multipole modes which are comparable with the  monopole mode and hence invalidate the
monopole point interaction approximation. We leave this case as an open question for future investigation. 
\end{rmk}

\begin{rmk}
Assumptions \ref{assump0-1} and \ref{asump0-4} are important in our effective medium theory. 
The parameter $\epsilon_1$ in Assumption \ref{assump0-1} controls the deviation of the frequency from the Minnaert resonant frequency, which further controls the amplitude of the scattering strength of each bubble. This parameter, together with $\Lambda$, also controls the volume fraction of the bubbles through Assumption \ref{asump0-4}.  In an informal way, if the bubble volume fraction is below the level as set by Assumption \ref{asump0-4}, say $s^{1-\epsilon_3}\cdot N =O(1)$ for some $\epsilon_3 < \epsilon_1$, then the effect of the bubbles is negligible and the effective medium would be the same as if there are no bubbles in  the limit as $N \to \infty$. On the other hand, if 
$s^{1-\epsilon_3}\cdot N =O(1)$ for some $\epsilon_3 > \epsilon_1$, then the bubbles interact strongly with 
each other and eventually behave as a medium with infinite effective refractive index. 
Only at the appropriate volume fraction as in Assumption \ref{asump0-4}, we have an effective medium theory with finite refractive index. The larger $\Lambda$ is, the higher the effective refractive index is. These statements can be justified by
the method developed in the paper.  
\end{rmk}

\begin{rmk}
One can easily check that Assumptions \ref{assump0}, \ref{assump2} and \ref{assump1} hold for
periodically distributed $y_j^N$'s. 
\end{rmk}


\section{Justification of the point interaction approximation} \label{sec-point interaction}
In this section, we justify the point interaction approximation under the assumptions we made in the previous section. Our main result is the following. 

\begin{prop} \label{prop-point-approx}
Under Assumptions \ref{assump0-1}, \ref{assump0}, \ref{assump2} and \ref{assump0-3}, 
the following 
relation between $u^{s,N}_j$ and $u^{i,N}_j$ holds for all $x$ such that $|x-y_j^N| \gg s$:
\beas
u^{s,N}_j(x)
= G(x, y_0, k) \cdot g \cdot \left( u_j^{i,N}(y_j^N) + O[N^{\f{\epsilon_0}{3} -\f{\epsilon_1}{1-\epsilon_1}}
+\f{s}{|x-y_j^N|} ]  \cdot \max_{1\leq l \leq N} |u_j^{i,N}(y_j^N)| \right). 
\eeas
Moreover, for $x=y_j^N$, 
$$
u_j^{i,N} (y_j^N) = u^i(y_j^N) + \sum_{i \neq j} u^{s,N}_i(y_j^N) 
=u^i(y_j^N) + \sum_{i \neq j} g \cdot \left(u^{i,N}_i(y_i^N) + p_i^N\right) G(y_j^N,y_i^N, k),
$$
where  
\beas
g&=&g(\omega, \delta, D_j^N) = -\f{s Cap(B)}{ 1- (\f{\omega_M}{\omega})^2 + i\gamma} (1+ O(s)+O(\delta)),\\
\gamma &=& \f{(\tau +1)v Cap(B)s\omega}{8 \pi} - \f{(\tau-1)Cap(B)^2 \delta}{8 \pi \tau^2 v Vol(B) \omega s},
\eeas
are the scattering and damping coefficients near the Minnaert resonant frequency respectively, 
and $p^N_i$ satisfies 
\[
|p^N_i| = \max_{1\leq i \leq N}|u^{i,N}_i(y_i^N)| \cdot O(N^{\f{\epsilon_0}{3} -\f{\epsilon_1}{1-\epsilon_1}}).
\]

\end{prop}

\begin{proof}
We follow the proof of Theorem 3.1 in \cite{hai3}.
Some modifications are needed because of two main differences between the case considered in 
\cite{hai3} and here: 
the first difference is that the bubble size 
is normalized to be of order one in \cite{hai3} while it is of order $s \ll 1$ here, the other one is that the incident field is assumed be to a plane wave in \cite{hai3} while this is no longer valid here. However, the approach still applies. 

First, by Taylor series expansion of $G(x, y, k)$ with respect to $y$ around $y_j^N$, we have  
\be
\label{numberq}
\begin{array}{lll}
u^{s,N}_j(x) &=& \ds \int_{\p D_j^N} G(x, y, k) \psi_j^N(y) d\sigma(y) \\
\nm
&=& \ds G(x, y_j^N, k) \left( (\chi_{\p D_j^N}, \psi_j^N)_{L^2} +O(\f{s}{|x-y_j^N|})\cdot s \cdot\|\psi_j^N\|_{L^2} \right).
\end{array}
\ee
On the other hand, from the argument in \cite{hai3}, one can obtain
\[
\psi_j^N= u_j^{i,N}(y_j^N) \mathcal{S}_{D_j^N}^{-1} [\chi_{\p D_j^N}] \cdot \f{g}{ Cap(D_j^N)}
+ \f{1}{s} \cdot O(\|F_{j,2}\|_{H^1(\p D_j^N)}), \quad \mbox{in}\,\, L^2(\p D_j^N),
\]
where 
$$
F_{j,2} (y) = u^{i,N}_j(y) -   u^{i,N}_j(y_j^N) = \sum_{i \neq j} \left(\mathcal{S}_{D_i^N}^k[\psi_i^N](y) -\mathcal{S}_{D_i^N}^k[\psi_i^N](y_j^N) \right).  
$$ 
By Lemma \ref{lem-tech-11}, we get
$$
\|\mathcal{S}_{D_j^N}^k(\psi_i^N)(y) -\mathcal{S}_{D_j^N}^k(\psi_i^N)(y_j^N)\|_{H^1(\p D_j^N)} 
\lesssim \f{1}{|y-y_j^N|^2}  \cdot s^2 \cdot \| \psi_i^N\|_{L^2(\p D_j^N)}. 
$$
Thus, 
\[
\|F_{j,2}\|_{H^1(\p D_j^N)}  \lesssim \sum_{i\neq j} \f{1}{|y_i^N-y_j^N|^2}  \cdot s^2 \cdot \max_{1\leq l \leq N} \| \psi_i^N\|_{L^2(\p D_j^N)}.
\]
Therefore, it follows that
$$
\| \psi_j^N\|_{L^2(\p D_j^N)}  \lesssim |u_j^{i,N}(y_j^N)| \cdot \| \mathcal{S}_{D_j^N}^{-1} [\chi_{\p D_j^N}] \|_{L^2(\p D_j^N)} \cdot |\f{g}{ Cap(D_j^N)}| + \sum_{i\neq j} \f{1}{|y-y_j^N|^2}  \cdot s \cdot \max_{1\leq l \leq N} \| \psi_l^N\|_{L^2(\p D_l^N)}.
$$
Note that $\| \mathcal{S}_{D_j^N}^{-1} [\chi_{\p D_j^N}] \|_{L^2(\p D_j^N)} =O(1)$ and 
\[
\sum_{i\neq j} \f{1}{|y_i^N-y_j^N|^2}  \cdot s  \lesssim r_N^{-\epsilon_0} s\cdot N \lesssim  N^{\f{\epsilon_0}{3} -\f{\epsilon_1}{1-\epsilon_1}},
\]
where we have used Assumption \ref{assump0-1} in the last inequality.
We can therefore conclude that
\be \label{estimate111}
\max_{1\leq j \leq N}\| \psi_j^N\|_{L^2(\p D_j^N)}  \lesssim \max_{1\leq j \leq N} |u_j^{i,N}(y_j^N)| \cdot |\f{g}{Cap(D_j^N)}|.
\ee
Consequently, by (\ref{numberq}), 
\beas
u^{s,N}_j(x)
&=& G(x, y_0, k) \left( (\chi_{\p D_j^N}, \psi_j^N)_{L^2} +O(\f{s}{|x-y_j^N|})\cdot s \cdot \|\psi_j^N\|_{L^2} \right) \\ 
&=& G(x, y_0, k) \left( (\chi_{\p D_j^N}, \psi_j^N)_{L^2} +O(\f{s}{|x-y_j^N|})\max_{1\leq j \leq N} |u_j^{i,N}(y_j^N)| \cdot |g| \right). 
\eeas 
Since 
\beas
(\chi_{\p D_j^N}, \psi_j^N)_{L^2} &=&\left(\chi_{\p D_j^N}, u_j^{i,N}(y_j^N) \mathcal{S}_{D_j^N}^{-1} [\chi_{\p D_j^N}] \cdot \f{g}{ Cap(D_j^N)} \right)_{L^2} + O(\|F_{j,2}\|_{H^1(\p D_j^N)}) \\
&=& u_j^{i,N}(y_j^N) g + O(s \cdot N^{\f{\epsilon_0}{3} -\f{\epsilon_1}{1-\epsilon_1}})  \cdot \max_{1\leq l \leq N} \| \psi_i^N\|_{L^2(\p D_j^N)} \\
&=& u_j^{i,N}(y_j^N) g + O(s \cdot N^{\f{\epsilon_0}{3} -\f{\epsilon_1}{1-\epsilon_1}})  \cdot \max_{1\leq j \leq N} |u_j^{i,N}(y_j^N)| \cdot |\f{g}{Cap(D_j^N)}|\\
&=& g \left(u_j^{i,N}(y_j^N) + O(N^{\f{\epsilon_0}{3} -\f{\epsilon_1}{1-\epsilon_1}})  \cdot \max_{1\leq l \leq N} |u_j^{i,N}(y_j^N)| \right),
\eeas
we arrive at
\beas
u^{s,N}_j(x)
= G(x, y_0, k) g \left( u_j^{i,N}(y_j^N) + O[N^{\f{\epsilon_0}{3} -\f{\epsilon_1}{1-\epsilon_1}}
+\f{s}{|x-y_j^N|} ]  \cdot \max_{1\leq l \leq N} |u_j^{i,N}(y_j^N)| \right).
\eeas
Finally, note that
$$
u^{i,N}_j(x) = u^i(x) + \sum_{i \neq j} u^{s,N}_i(x).
$$
By taking $x=x_i^N$ and using the assumption that
$$
|x_i^N-x_j^N| \geq r_N, 
$$
we obtain
$$
\f{s}{|x-y_j^N|} \leq \f{s}{r_N} \lesssim \f{1}{N} \cdot s^{\epsilon_1} \cdot N^{\f{1}{3}} \lesssim N^{\f{\epsilon_0}{3} -\f{\epsilon_1}{1-\epsilon_1}}.
$$
The second part of the proposition follows immediately. 
\end{proof}

\begin{lem} \label{lem-tech-11}
The following estimate holds:
\be \label{estimate111b}
\|\mathcal{S}_{D_j^N}^k(\psi_i^N)(y) -\mathcal{S}_{D_j^N}^k(\psi_i^N)(y_j^N)\|_{H^1(\p D_j^N)} 
\lesssim \f{1}{|y_i^N-y_j^N|^2}  \cdot s^2 \cdot \| \psi_i^N\|_{L^2(\p D_j^N)}. 
\ee
\end{lem}
\begin{proof}
By Taylor series expansion of  $G(y, z, k)$ with respect to $y$ around $y_j^N$ and $z$
around $y_i^N$, we have
\beas
\mathcal{S}_{D_j^N}^k(\psi_i^N)(y) -\mathcal{S}_{D_j^N}^k(\psi_i^N)(y_j^N)
 &= & \int_{\p D_j^N} \left(G(y, z, k) -G(y_j^N, z, k)\right) \psi_j^N(z) d\sigma(z)  \\ 
 &= & \sum_{|\alpha| \geq 1} (y-y_j^N)^{\alpha} \sum_{|\beta| \geq 0} \int_{\p D_i^N} \f{\p^{|\alpha|+ |\beta|} G}{\p y^{\alpha} z^{\beta}} (y_j^N, y_i^N, k) (z-y_i^N)^{\beta}
 \psi_i^N (z) d\sigma(z).
\eeas
Using the estimate
$$
| \f{\p^{|\alpha|+ |\beta|} G}{\p y^{\alpha} z^{\beta}} (y_j^N, y_i^N, k) | \lesssim
\max\{ \f{1}{|y_i^N- y_j^N|}, \f{1}{|y_i^N- y_j^N|^{|\alpha|+ |\beta|+1}} \},
$$
we obtain
\beas
|\mathcal{S}_{D_j^N}^k(\psi_i^N)(y) -\mathcal{S}_{D_j^N}^k(\psi_i^N)(y_j^N)| & \lesssim &
\f{1}{|y_i^N- y_j^N|^2} \cdot s^2 \cdot \| \psi_i^N\|_{L^2}, \\
| \nabla \mathcal{S}_{D_j^N}^k(\psi_i^N)(y)| & \lesssim &
\f{1}{|y_i^N- y_j^N|^2} \cdot s \cdot \| \psi_i^N\|_{L^2},
\eeas
whence estimate (\ref{estimate111b}) follows. This completes the proof. 

\end{proof}

Let us denote $x_j^N= u_j^{i,N}(y_j^N)$,  $b_j^N = u^i(y_j^N)$, $T^N=(T_{ij}^N)_{1\leq i, j \leq N}$ with $T_{ij}^N= g G(y_i^N,y_j^N,k)$,  and 
$$
q^N_j = \sum_{i \neq j} gG(y_j^N,y_i^N, k) p_i^N. 
$$ 
We obtain the following system of equations for
$x^N= (x_j^N)_{1\leq j \leq N}$: 
\be  \label{equation-xN}
x^N  - T^Nx^N 
=b^N +  q^N .
\ee

\section{Effective medium theory of bubbly media} \label{sec-effective media}
In this section we derive an effective medium theory for the acoustic wave propagation in
the bubbly fluid considered in Section \ref{sec-bubbly media}. We first establish the 
well-posedness, including existence, uniqueness and stability, and the limiting behavior of the solution to the system of equations (\ref{equation-xN}), which is resulted from the point interaction approximation, in Subsection \ref{subsec-effective-1}. We then construct wave field from the solution to 
(\ref{equation-xN}) and show the convergence of the constructed micro-field to a 
macro-effective field, in Subsection \ref{subsec-effective-2}.

\subsection{Well-posedness and limiting behaviour of the point interaction system} \label{subsec-effective-1}

We start from the summation $\sum_{i \neq j} g G(y_j^N, y_i^N,k) f(y_i^N)$.
It is clear that
\[
 \sum_{i \neq j} g G(y_j^N, y_i^N,k) f(y_i^N) =  \f{1}{N} \sum_{i \neq j} 
 \f{-Cap(B)}{ \beta_0 s^{\epsilon_1} + i\cdot O(\omega \cdot s)} (s\cdot N) \cdot G(y_j^N,y_i^N,k) f(y_i^N).
\]

Denote by
$$
\beta_N= \f{-Cap(B)}{\beta_0 + i\cdot O(\omega \cdot s^{1-\epsilon_1})}(1+ O(s)), \quad
\beta = \f{-Cap(B)}{\beta_0}.
$$

Note that $\beta$ and $B$ are independent of $N$. By Assumption \ref{asump0-4}, 
we have the following identity:
$$
\sum_{i \neq j} g G(y_j^N, y_i^N,k) f(y_i^N) =\f{1}{N} \sum_{i \neq j} 
 \beta_N \cdot \Lambda \cdot G(y_j^N,y_i^N,k) f(y_i^N).
$$

Let
\be  \label{v}
V(x) = \beta \cdot \Lambda \cdot \tilde{V}(x).
\ee

We note that there are two cases depending on whether $\omega > \omega_M$ or 
$\omega < \omega_M$. In the former case, $\beta_0 >0$, thus $\beta >0$ which leads to 
$V(x) \geq 0$, while in the 
latter case we have $\beta_0 <0$ and thus $\beta <0$ which leads to 
$V(x) \leq 0$.

We now present a result on the approximation of the summation $\sum_{i \neq j} g G(y_j^N, y_i^N,k) f(y_i^N)$ by using volume integrals. 

\begin{lem} \label{lem-asump}
For any $f \in C^{0, \alpha}(\Omega)$ with $0 < \alpha \leq 1$,
\beas
\max_{1\leq j \leq N} | \f{1}{N} \sum_{i \neq j} \beta_N \cdot \Lambda \cdot G(y_j^N,y_i^N,k) f(y_i^N) -
\int_{\Omega} G(y_j^N, y,k) V(y) f(y) dy| \lesssim \f{1}{N^{\f{\alpha}{3}}}\|f\|_{C^{0, \alpha}(\Omega)}.  
\eeas
\end{lem} 

\begin{proof}
By Assumption \ref{assump1}, we have
\beas
\max_{1\leq j \leq N} | \f{1}{N} \sum_{i \neq j} \beta \cdot \Lambda \cdot G(y_j^N,y_i^N,k) f(y_i^N) -
\int_{\Omega} G(y_j^N, y,k) V(y) f(y) dy| \lesssim \f{1}{N^{\f{\alpha}{3}}}\|f\|_{C^{0, \alpha}(\Omega)}.  
\eeas
On the other hand, 
note that 
$$
|\beta_N - \beta| \lesssim s^{1-\epsilon_1} \lesssim \f{1}{N}. 
$$
Thus, 
\beas
\max_{1\leq j \leq N} | \f{1}{N} \sum_{i \neq j} \left(\beta -\beta_N\right) \cdot \Lambda \cdot G(y_j^N,y_i^N,k) f(y_i^N)| & \lesssim & \f{1}{N} \cdot \f{1}{N}
\sum_{i \neq j} \f{1}{|y_j^N -y_i^N|} \|f\|_{C^{0, \alpha}(\Omega)} \\
& \lesssim &  \f{1}{N} \|f\|_{C^{0, \alpha}(\Omega)} \leq  \f{1}{N^{\f{\alpha}{3}}}\|f\|_{C^{0, \alpha}(\Omega)}.
\eeas
The lemma then follows immediately.

\end{proof}

Let $X= C^{0, \alpha}(\Omega)$ for some $0 < \alpha <1$ (later on  we will take $\alpha = \f{1-\epsilon_0}{2}$). Define $\mathcal{T}$ by
$$
\mathcal{T} f(x) = \int_{\Omega} G(x, y, k) V(y) f(y)dy.
$$
$\mathcal{T}$ can be viewed as the continuum limit of $T^N$ in some sense. One can show that 
$\mathcal{T}: X \to X$ is a compact linear operator. Moreover, the following properties hold. 

\begin{lem}\label{lem-limit-operator}
\begin{enumerate}
\item[(i)]
The operator $\mathcal{T}$ is bounded from $C^0({\overline{\Omega}})$ to $C^{0, \alpha}(\Omega) $ for any $0< \alpha <1$. 

\item[(ii)]
The operator $\mathcal{T}$ is bounded from $C^{0, \alpha}(\Omega)$ to $C^{1, \alpha}(\Omega) $ for any $0< \alpha <1$. 

\item[(iii)]
In the case when $\omega < \omega_M$, the operator $Id - \mathcal{T}$ has a bounded inverse on the Banach space $X$. More precisely, for each $b \in X$, there exists a unique $f \in X$ such that $f - \mathcal{T}f =b$ and $\|f\|_{X} \leq C \| f\|_X$, where $C$ is a positive constant independent of $b$. 

\item[(iv)]
In the case when $\omega > \omega_M$, the same conclusion as in Assertion (iii) holds, provided that 
$V(x)> k^2$ almost everywhere in $\Omega$. 
\end{enumerate}
\end{lem}

\begin{proof}
Assertions (i) and (ii) follow from the general theory on integral operators in \cite{colton}. We now show Assertion (iii).
Let 
$b \in X$ and consider the following integral equation 
\[
x - \mathcal{T}x = b.
\]
Applying the operator $\triangle + k^2$ to both sides of the above equation, we obtain
$$
(\triangle + k^2) x - V x =(\triangle + k^2)b \,\,\, \mbox{in} \,\, \Omega. 
$$
In the case when $\omega < \omega_M$, we have $V(x) \leq 0$. Thus the above equation yields a Lippmann-Schwinger equation with potential $k^2- V$, for which the solution is known to be unique.  This proves that the operator 
$Id - \mathcal{T}$ has a trivial kernel. The rest of statements of Assertion (iii) follow from standard Fredholm theory for compact operators.  Similarly, for Assertion (iv), we note that the operator
$ \triangle + k^2 - V $ is elliptic, then the statement follows from the standard theory of elliptic equations. 
\end{proof}

\begin{rmk}
\label{rem41}
In the case when $\omega > \omega_M$, one has $V(x) \geq 0$.  The integral equation $x - \mathcal{T}x = b$
leads to the following partial differential operator $ \triangle + (k^2 - V)$ where $k^2-V$ may change sign in the domain $\Omega$ depending on the values of $V(x)$. In fact, in some physical situations, $\tilde{V}$ may be zero or negligible near $\p \Omega$ while of order one inside $\Omega$. When $\beta\cdot \Lambda \gg 1$, we 
see that $k^2-V <0$ in the inner region of $\Omega$. As a consequence, wave field is attenuating therein, which implies that the effective medium is dissipative. On the other hand, the wave field is still propagating near $\p \Omega$ where $k^2-V(x)$ is positive. One may also see a transition layer from propagating region  to dissipative region near the place when $k^2-V(x)$ is close to $0$. It is not clear whether the operator $\triangle + (k^2 - V)$
with $k^2-V$ changing sign is uniquely solvable or not. 
\end{rmk}

In view of Remark \ref{rem41}, we shall restrict our investigation to the case when  $\omega < \omega_M$ from now on. However, we remark that if we assume that kernel of the operator $Id -\mathcal{T}$ is trivial in the case when $\omega > \omega_M$, then all the arguments and results which hold for the case $\omega < \omega_M$ also hold for $\omega > \omega_M$.

Note that
$u^i \in X$. Let $\psi$ be the unique solution satisfying
\be  \label{eq-psi}
\psi - \mathcal{T}\psi = u^i.
\ee 
It is clear that
$$
(\triangle + k^2) \psi - V \psi =0 \,\,\, \mbox{in} \,\, \R^3. 
$$
 
We shall show that $\psi$ is the limit of the solution $x^N$ to (\ref{equation-xN}) in a sense which will be made clear later on. 
We first present the following result concerning the well-posedness of the discrete system (\ref{equation-xN}).  
\begin{prop} \label{prop1}
Let $X = C^{0, \alpha}(\Omega)$ for $\alpha =\f{1-\epsilon_0}{2}$ and assume that $\omega < \omega_M$.
Then under Assumptions \ref{assump0}, \ref{assump2} and \ref{assump1}, there exists $N_0>0$ such that for all $N \geq N_0$ and $b \in X$, there is a unique solution to the equation
$$
 z^N- T^N z^N = b^N
$$ 
with $b^N_j = b(y_j^N)$. 
Moreover, 
$$
\max_{1\leq j \leq N}| z^N_j| \leq C_1 \| b\|_X,
$$
for some constant positive $C_1$ independent of $N$ and $b$. 
\end{prop}
The proof of this proposition is technical and is postponed to Section \ref{sec-proof}.
As a corollary of the proposition, we can prove our main result on the limiting behavior of the solution to the system (\ref{equation-xN}).

\begin{thm} \label{thm1}
Let $X = C^{0, \alpha}(\Omega)$ for $\alpha =\f{1-\epsilon_0}{2}$ and assume that $\omega < \omega_M$. Then under Assumptions  \ref{assump0-1}, \ref{asump0-4}, \ref{assump0}, \ref{assump2}, \ref{assump1} and \ref{assump0-3}, there exists $N_0>0$ such that for all $N \geq N_0$,
\[
\max_{1\leq j \leq N} |x_j^N-\psi(y_j^N)| \lesssim N^{-\f{1-\epsilon_0}{6}},
\]
where $x^N$ and $\psi$ are the solutions to (\ref{equation-xN}) and (\ref{eq-psi}), respectively. 
\end{thm}

\begin{proof} 
Step 1. 
We have
\beas
x_j^N - \f{1}{N} \sum_{i \neq j } \beta_N \cdot \Lambda \cdot G(y_j^N, y_i^N,k) x_i^N &= & b_j^N + q_j^N, \\
\psi(y_j^N) - \int_{\Omega} G(y_j^N, y, k) V(y)f(y) dy &= & b_j^N.
\eeas
Let $r_j^N= x_j^N - \psi(y_j^N)$. Then
\[
r^N - T^n r^N=   e^N + q^N,
\]
where
$$
e^N_j = \f{1}{N} \sum_{i \neq j } \beta_N \cdot \Lambda \cdot G(y_j^N, y_i^N,k) \psi(y_j^N) - \int_{\Omega} G(y_j^N, y, k) V(y)\psi(y) dy.
$$

Step 2. Let $G_N(x, y)$ be defined as in Step 1 in the proof of Lemma \ref{lem-epsilon} and define
$$
G_N(x,y,k) = G_N(x, y) + (G(x,y, k) - G(x,y,0)):= G_{N,1}(x,y,k)+G_{N,2}(x,y,k).
$$

Denote by
$$
\tilde{q}^N(y) = \sum_{i \neq j} gG_N(y,y_i^N, k) p_i^N =\tilde{q}_1^N(y) + \tilde{q}_2^N(y),
$$ 
where 
$$
\tilde{q}_1^N(y) = \sum_{i \neq j} gG_{N,1}(y,y_i^N, k) p_i^N, \quad 
\tilde{q}_2^N(y) = \sum_{i \neq j} gG_{N,2}(y,y_i^N, k) p_i^N.
$$ 
By Lemma \ref{lem-technique2}, 
$\tilde{q}_1^N  \in X$. Moreover, 
$$
\|\tilde{q}_1^N\|_X \lesssim  \max_{1\leq i \leq N} |p^N_i| \lesssim 
O(N^{\f{\epsilon_0}{3} -\f{\epsilon_1}{1-\epsilon_1}})\cdot \max_{1\leq i \leq N} |x_i^N|. 
$$
Since $G_{N,2}(x,y,k)$ is smooth in $|x-y|$ and is bounded, a straightforward calculation shows that
$\tilde{q}_2^N  \in X$ as well and 
$$
\|\tilde{q}_2^N\|_X \lesssim 
O(N^{\f{\epsilon_0}{3} -\f{\epsilon_1}{1-\epsilon_1}})\cdot \max_{1\leq i \leq N} |x_i^N|. 
$$
Thus, we have
$\tilde{q}^N  \in X$ and
$$
\|\tilde{q}^N\|_X \lesssim  \max_{1\leq i \leq N} |p^N_i| \lesssim 
O(N^{\f{\epsilon_0}{3} -\f{\epsilon_1}{1-\epsilon_1}})\cdot \max_{1\leq i \leq N} |x_i^N|. 
$$

On the other hand,
by Lemma \ref{lem-epsilon}, there exists $\tilde{e}^N \in X$ such that
$\tilde{e}^N(y^N_j)= e^N_j$ and $\|\tilde{e}^N \|_X \lesssim N^{-\f{1-\epsilon_0}{6}}\|u^i\|_X$. 
Therefore, 
\[
\|\tilde{e}^N\|_X + \|\tilde{q}^N \|_X  \lesssim  N^{-\f{1-\epsilon_0}{6}}\|u^i\|_X + O(N^{\f{\epsilon_0}{3} -\f{\epsilon_1}{1-\epsilon_1}})\cdot \max_{1\leq i \leq N} |x_i^N| .
\]
It then follows from 
Proposition \ref{prop1} that, 
\be  \label{ineq-1}
\max_{1\leq j \leq N}|r^N_j|  \lesssim N^{-\f{1-\epsilon_0}{6}}\|u^i\|_X + O(N^{\f{\epsilon_0}{3} -\f{\epsilon_1}{1-\epsilon_1}})\cdot \max_{1\leq i \leq N} |x_i^N|. 
\ee

Step 3. Note that $\max_{1\leq j \leq N}|\psi(y_j^N)| $ is bounded independently of $N$. We can derive from (\ref{ineq-1}) that $\max_{1\leq j \leq N}|x_j^N| $ is also bounded independently of $N$, which further implies that
$$
\max_{1\leq j \leq N}|r^N_j|  \lesssim   N^{-\min \{ \f{1-\epsilon_0}{6}, \f{\epsilon_0}{3} -\f{\epsilon_1}{1-\epsilon_1} \}}. 
$$
This completes the proof of the theorem.  

\end{proof}

As a consequence of the above result and (\ref{estimate111b}), we have the following corollary.  
\begin{cor}
The following estimate holds:
$$
\max_{1\leq j \leq N}\| \psi_j^N\|_{L^2(\p D_j^N)}  \lesssim s^{-\epsilon_1} \cdot  \|u^i\|_X .
$$
\end{cor}

\subsection{Convergence of micro-field to the effective one}\label{subsec-effective-2}

Let us consider the total field $u^N= u^{in} + \sum_{1\leq j \leq N}\mathcal{S}_{D^N_j}^{k} [\psi^N_j]$ outside the bubbles. Define
\be
\tilde{u}^N(x) = u^i(x) + \sum_{1\leq j \leq N} g G(x,y_j^N, k) x_j^N 
= u^i(x) + \f{1}{N}\sum_{1\leq j \leq N} \beta_N \cdot \Lambda \cdot G(x,y_j^N, k) x_j^N, 
\ee
and denote by
$$
Y^N_{\epsilon_2} = \{ x: |x-y_j^N| \geq \f{1}{N^{1-\epsilon_2}} \,\,\mbox{for all} \,\,1\leq j\leq N\}
$$
for some fixed constant 
$\epsilon_2 \in (0, \f{1}{3})$. The reason for us to introduce the set $Y^N_{\epsilon_2}$ is that 
the convergence of micro-field to the effective field does not hold near the bubbles because of the singularity of the Green function near the source point. However, it holds in the region away from the bubbles, which is characterized by $Y^N_{\epsilon_2}$. 

\begin{lem}
The following estimate holds uniformly for all $x \in Y^N_{\epsilon_2}$:
$$
| \tilde{u}^N(x) - u^N(x)| \lesssim N^{\f{\epsilon_0}{3} -\f{\epsilon_1}{1-\epsilon_1}}.
$$
\end{lem}

\begin{proof}
For each $x \in Y^N_{\epsilon_2}$, 
it is clear that
$$
u^N(x) = u^i(x) + \sum_{1\leq j \leq N } u^{s,N}_j(x). 
$$
By Proposition \ref{prop-point-approx}, we have
\beas
u^N(x)&= & u^i(x)+ \sum_{1\leq j \leq N} gG(x, y_j^N, k)  \left( u_j^{i,N}(y_j^N) + O[N^{\f{\epsilon_0}{3} -\f{\epsilon_1}{1-\epsilon_1}}
+N^{-\f{\epsilon_1}{1-\epsilon_1}-\epsilon_2}]  \cdot \max_{1\leq l \leq N} |u_j^{i,N}(y_j^N)| \right) \\
& =& \tilde{u}^N(x) +  \sum_{1\leq j \leq N} gG(x, y_j^N, k)  \cdot O[N^{\f{\epsilon_0}{3} -\f{\epsilon_1}{1-\epsilon_1}}
 ]  \cdot \max_{1\leq l \leq N} |u_j^{i,N}(y_j^N)| \\
& =& \tilde{u}^N(x) +  \sum_{1\leq j \leq N} gG(x, y_j^N, k)  \cdot O[N^{\f{\epsilon_0}{3} -\f{\epsilon_1}{1-\epsilon_1}}
]  \cdot \|u^i\|_X \\ 
& =& \tilde{u}^N(x) +  \sum_{1\leq j \leq N} gG(x, y_j^N, k)  \cdot O[N^{\f{\epsilon_0}{3} -\f{\epsilon_1}{1-\epsilon_1}}
]  .
\eeas 
On the other hand,
\beas
\sum_{1\leq j \leq N} |g G(x, y_j^N, k)| &=& \f{1}{N}\sum_{1\leq j \leq N} |\beta_N| \cdot \Lambda \cdot |G(x, y_j^N, k)|\\
& \lesssim &  \f{1}{N} \cdot \sum_{1\leq j \leq N} \f{1}{|x-y_j^N|} \\
& \lesssim &  \f{1}{N} \max_{1\leq j \leq N} \f{1}{|x-y_j^N|} + 
\f{1}{N} \cdot \sum_{2r_N \leq |x-y_j^N|} \f{1}{|x-y_j^N|}
 \lesssim 1. 
\eeas 
Therefore, 
$$
u^N(x)= \tilde{u}^N(x) + N^{\f{\epsilon_0}{3} -\f{\epsilon_1}{1-\epsilon_1}}.
$$
This completes the proof of the Lemma.

\end{proof}

Define 
$$
w(x)= u^i(x)+ \int_{\Omega} G(x, y, k) V(y)\psi(y) dy. 
$$
We have the following two results. 

\begin{lem}
For all $x \in Y^N_{\epsilon_2}$, the following estimate holds uniformly: 
\[
| \tilde{u}^N(x) - w(x) |  \lesssim  N^{-\min\{ \f{1-\epsilon_0}{6}, \f{1- \epsilon_2}{3}, \epsilon_2, \f{\epsilon_0}{3} -\f{\epsilon_1}{1-\epsilon_1} \}}.
\]
\end{lem}

\begin{proof}
For each $x \in Y^N_{\epsilon_2}$, choose $y_l^N \in \{ y_j^N\}_{1\leq j \leq N}$ such that
\[
| x- y_l^N| =\min_{1\leq j \leq N} |x- y_j^N|.
\]
We have
\beas
\tilde{u}^N(x) - w(x) &= &  \f{1}{N} \beta_N \cdot \Lambda \cdot G(x,y_l^N, k) x_l^N + \f{1}{N}\sum_{j \neq l} \beta_N \cdot \Lambda \cdot G(x,y_j^N, k) x_j^N -\int_{\Omega} G(x, y, k) V(y)\psi(y) dy  \\
& =&  \left[\f{1}{N}\sum_{j \neq l} \beta_N \cdot \Lambda \cdot G(y_l^N,y_j^N, k) \psi(y_j^N) -\int_{\Omega} G(y_l^N, y, k) V(y)\psi(y) dy  \right] \\
&& + \f{1}{N}\sum_{j \neq l} \beta_N \cdot \Lambda \cdot [G(x,y_j^N, k) - G(y_l^N,y_j^N, k)] \psi(y_j^N)  \\
&& + \f{1}{N}\sum_{j \neq l} \int_{\Omega} [G(x, y, k) -G(y_l^N, y, k)] V(y)\psi(y) dy  \\
&& +  \f{1}{N}\sum_{j \neq l} \beta_N \cdot \Lambda \cdot G(x,y_j^N, k) (x_j^N- \psi(y_j^N)) +\f{1}{N} \beta_N \cdot \Lambda \cdot G(x,y_l^N, k) x_l^N \\
& =:& e_1 + e_2 + e_3 + e_4 +e_5.
\eeas

Let us now estimate $e_j$, $j=1,\cdots, 5$ one by one. 

First, by Assumption \ref{assump1}, 
$$
|e_1| \lesssim N^{-\f{\alpha}{3}}\cdot \|\psi\|_X \lesssim N^{-\f{1-\epsilon_0}{6}}.
$$

Second, similar to Lemma \ref{lem-technique2}, we can show that
$$
|e_2|  \lesssim |x-y_l^N|^{1- \epsilon_2} \cdot \|\psi\|_X \lesssim N^{-\f{1- \epsilon_2}{3}}\|\psi\|_X .
$$

Third, by Lemma \ref{lem-limit-operator}, 
$$
|e_3|  \lesssim |x-y_l^N|^{1- \epsilon_2} \cdot \|\psi\|_X \lesssim N^{-\f{1- \epsilon_2}{3}}\|\psi\|_X .
$$

Fourth, note that 
$$
|e_4| \lesssim
 \f{1}{N}\sum_{j \neq l} |\beta_N| \cdot \Lambda \cdot | \max_{1\leq j \leq N}|x_j^N- \psi(y_j^N)| \cdot \f{1}{|x-y_j^N|}.
$$
By Assumption \ref{assump2} and Theorem \ref{thm1}, we have
$$
|e_4| \lesssim \max_{1\leq j \leq N}|x_j^N- \psi(y_j^N)| \lesssim N^{-\min\{ \f{1-\epsilon_0}{6}, \f{\epsilon_0}{3} -\f{\epsilon_1}{1-\epsilon_1}\}}. 
$$
Finally, one can check that
$$
|e_5| \lesssim \f{1}{N} \cdot N^{1-\epsilon_2} \cdot \max_{1\leq j \leq N}\| x_j^N\| 
 \lesssim N^{-\epsilon_2}. 
$$
Therefore,  
$$
\tilde{u}^N(x) - w(x) = O( N^{-\min\{ \f{1-\epsilon_0}{6}, \f{1- \epsilon_2}{3}, \epsilon_2, \f{\epsilon_0}{3} -\f{\epsilon_1}{1-\epsilon_1} \}}). 
$$
This complete the proof of the Lemma.

\end{proof}
The following lemma holds. 
\begin{lem} \label{lem-w}
We have $w= \psi$.
\end{lem}

\begin{proof}
It is clear that $w$ satisfies the equation
$$
(\triangle + k^2) w =  (\triangle + k^2) u^i +  V \psi = V \psi. 
$$
Recall that
$$
(\triangle + k^2) \psi - V \psi =(\triangle + k^2) u^i=0. 
$$
Therefore, we have
$$
(\triangle + k^2) (w- \psi) =0. 
$$
On the other hand, it is easy to see the $w-\psi$ satisfies the radiation condition. The conclusion $w= \psi$ follows immediately. 
\end{proof}

As a consequence of the above two lemmas, we obtain the following theorem. 

\begin{thm}
Let $\omega < \omega_M$ and let $V$ be defined by (\ref{v}). Then under Assumptions \ref{assump0-1}--\ref{assump0-3}, the solution to the scattering problem
(\ref{eq-scattering}) converges to the solution to the wave equation
$$ 
(\triangle + k^2-V) \psi  =0
$$
together with the radiation condition imposed on $\psi- u^i$ at infinity, in the sense that
for $x \in Y^N_{\epsilon_2}$, the following estimate holds uniformly:
\[
| u^N(x) - \psi(x) |  \lesssim N^{-\min\{ \f{1-\epsilon_0}{6}, \f{1- \epsilon_2}{3}, \epsilon_2, \f{\epsilon_0}{3} -\f{\epsilon_1}{1-\epsilon_1} \}}.
\]

\end{thm}

The above theorem shows that under certain conditions, we can treat the bubbly fluid as an effective medium for acoustic wave propagation. Note that 
$$
\triangle + k^2 - V = \triangle + k^2 (1- \f{1}{k^2}\beta \cdot \Lambda \cdot \tilde{V}). 
$$
Thus, the effective medium can be characterized by the refractive index 
$1- \f{1}{k^2}\beta \cdot \Lambda \cdot \tilde{V}$. 
By our assumption, $k=O(1)$ and $\tilde{V} = O(1)$. When $\beta \cdot \Lambda \gg 1$, we see that we have an effective high refractive index medium. As a consequence, this together with the main result in \cite{hai2} gives a rigorous mathematical theory for the super-focusing experiment in \cite{lanoy}.  

Similarly, we have the following result for the case $\omega > \omega_M$.

\begin{thm}
Let $\omega > \omega_M$ and assume that $V(x)> k^2$ almost everywhere in $\Omega$. Then under Assumptions \ref{assump0-1}-\ref{assump0-3}, the solution to the scattering problem
(\ref{eq-scattering}) converges to the solution to the following dissipative equation
$$ 
(\triangle + k^2 -V) \psi=0
$$
together with the radiation condition imposed on $\psi- u^i$ at infinity, in the sense that
for $x \in Y^N_{\epsilon_2}$, the following estimate holds uniformly:
\[
| u^N(x) - \psi(x) |  \lesssim N^{-\min\{ \f{1-\epsilon_0}{6}, \f{1- \epsilon_2}{3}, \epsilon_2, \f{\epsilon_0}{3} -\f{\epsilon_1}{1-\epsilon_1} \}}.
\]

\end{thm}

Finally, we conclude this section with the following three important remarks.

\begin{rmk}
At the resonant frequency $\omega= \omega_M$, the scattering coefficient $g$ is of order one. 
Thus each bubble scatter is  a point source with magnitude one.  As a consequence, the addition or removal of one bubble from the fluid affects the total field by a magnitude of the same order as the incident field. Therefore, we cannot expect any effective medium theory for the bubbly medium at this resonant frequency.  
\end{rmk}

\begin{rmk}
The super-focusing (or equivalently super-resolution) theory, developed in this paper for bubbly fluid seems to be 
different from the one developed for Helmholtz resonators \cite{hai} and plasmonic 
nanoparticles \cite{matias}. However, they are closely related. In \cite{hai, matias}, it is shown that super-focusing (or super-resolution) is due to sub-wavelength propagating resonant modes which are generated by the sub-wavelength resonators embedded in the background homogeneous medium. In those two cases, the region  with subwavelength resonators has size smaller or much smaller than the incident wavelength, and the number of sub-wavelength resonators is not very large, so is the number of sub-wavelength resonant modes. As a result, an effective medium theory is not necessary or even true. However, in the case of bubbles in a fluid as considered in this paper, the region with bubbles has size comparable or greater than the incident wavelength. This together with the fact that the ratio between the size of the individual bubble and the incident wavelength near the Minnaert resonant frequency is extremely small, indicates that the number of bubbles can be very large as is in the experiment in \cite{lanoy}, even though they are dilute. This large number of bubbles generates a large number of resonant modes which eventually yield a continuum limit in the form of an effective medium with high refractive index. In fact, these resonant modes can be obtained from the point interaction system 
(\ref{equation-xN}). On the other hand, it is shown in \cite{hai3} that super-focusing (or super-resolution) is possible in high refractive index media. In this regard, the effective medium theory developed in this paper can be viewed as a bridge between the super-focusing (or super-resolution) theories in \cite{hai} and \cite{hai3}. 
\end{rmk}

\begin{rmk}
In this section, we derived an effective medium theory for the case  
$\omega  < \omega_M$ and a special case of $\omega  > \omega_M$ with some additional assumptions.  However, our results still hold for the case $\omega  > \omega_M$ without any additional assumption, if we assume that the limiting system $Id - \mathcal{T}$ has a trivial kernel. This assumption implies that the limiting system is well-posed. 
\end{rmk}

\section{Proof of Proposition \ref{prop1}}  \label{sec-proof}
\begin{proof}
Step 1. Let $b\in X$ be given and let $\psi$ be the solution to 
$$
(Id - \mathcal{T}) \psi = b. 
$$
By Lemma \ref{lem-limit-operator}, we have
$\|\psi\|_X  \leq C_1 \|b\|_X$ for some constant $C_1$ independent of $N$.
Note that
$$
\psi(y_j^N) -\int_{\Omega}G(y_j^N, y, k) V(y) \psi(y) dy = b(y_j^N).
$$

Denote $\psi_j^{N,0}= \psi(y_j^N)$, $b_j^N = b(y_j^N)$. 
Then we have
$$
\psi^{N,0} -T^N \psi^{N,0} = b^N + \epsilon^{N,0},
$$ 
where 
\be \label{epsilon1}
\epsilon_j^{N,0} = \f{1}{N} \sum_{i \neq j}\beta_N \cdot \Lambda \cdot G(y_j^N,y_i^N,k) b(y_j^N) -
\int_{\Omega} G(y_j^N, y,k) V(y) b(y) dy.
\ee

It is clear that 
\beas
\max_{1\leq j \leq N} |\psi_j^{N,0}|  & \leq & C_1 \|b\|_X, \\
\max_{1\leq j \leq N} |\epsilon_j^{N,0}| &\leq & C_2 N^{-\alpha} \|b\|_X,
\eeas
for some constants $C_1$ and $C_2$ independent of $N$. 
Here, we have used Assumption \ref{assump2} in the second estimate above.  
By Lemma \ref{lem-epsilon} (which is given below this proposition), 
there exist a constant $C_3$ independent of $N$ and a function $\tilde{\epsilon}^{N,0} \in X$ such that
$\tilde{ \epsilon}^{N,0}(y_j^N) = \epsilon_j^{N,0}$ and 
$$
 \|\tilde{\epsilon}^{N,0}\|_X \leq C_3 \delta(N) \|b\|_X,   
$$
where $\delta(N)= N^{-\f{1-\epsilon_1}{6}}$.

Step 2. Let $b$ be replaced by $-\tilde{\epsilon}^{N,0}$. By applying the same argument as in Step 1, we can find $\psi^{N, 1}$ and $\tilde{\epsilon}^{N,1} \in X$ such that
$$
\psi^{N, 1} -T^N \psi^{N, 1} = - \epsilon^{N,0} + \epsilon^{N,1}, 
$$
with $\epsilon^{N,1}_j = \tilde{\epsilon}^{N,1}(y_j^N)$ and
\beas
\max_{1\leq j \leq N} |\psi_j^{N,1}|  &\leq & C_1 \|\tilde{\epsilon}^{N,0}\|_X \leq
C_1\cdot C_3\cdot \delta(N) \cdot \|b\|_X,\\
\|\tilde{\epsilon}^{N,1}\|_X & \leq &  C_3 \cdot \delta(N) \cdot \|\tilde{\epsilon}^{N,0}\|_X 
\leq (C_3\cdot \delta(N))^2 \cdot \|b\|_X. 
\eeas

Step 3. 
Continuing the procedure, we obtain sequences $\psi^{N, m}$, 
$\epsilon^{N,m}$ and $\tilde{\epsilon}^{N,m} \in X$, $m=1, 2, ...$ such that 
$$
\psi^{N, m} - T^N \psi^{N, m} = - \epsilon^{N,m-1} + \epsilon^{N,m}, 
$$
with $\epsilon^{N,m}_j = \tilde{\epsilon}^{N,m}(y_j^N)$ and
\beas
\max_{1\leq j \leq N} |\psi_j^{N,m}| & \leq & C_1 \cdot (C_3 \cdot \delta(N))^{m} \|b\|_X, \\
\|\tilde{\epsilon}^{N,m}\|_X &\leq &(C_3 \cdot \delta(N))^{m+1}\|b\|_X. 
\eeas

Step 4. By taking $N$ to be sufficiently large, say, $N\geq N_0$ for some $N_0 >0$, we can conclude that the series 
$
\sum_{m \leq 0 } \psi^{N, m} 
$
is absolutely convergent. We denote by $ \psi^N =\sum_{m \geq 0 } \psi^{N, m}$. 
Then
$$
\psi^N - T^N \psi^N = b^N,
$$
and
$$ 
\max_{1\leq j \leq N} |\psi_j^{N}|  \lesssim  \|b\|_X. 
$$

Step 5. So far, 
we have constructed a solution to the equation
$$
x- T^N x = b^N
$$
in the case when $b^N_j= b(y_j^N)$ for some $b\in X$. By varying $b$, we can show that
the solution exists for any $b^N$. It follows that the 
operator $Id - T^N$ is surjective. As a consequence, it is also injective. This proves the uniqueness of the solution and completes the proof of the proposition. 
\end{proof}

The following result is needed in the proof of the above proposition.

\begin{lem} \label{lem-epsilon}
Let $b\in X = C^{0, \alpha}(\Omega)$ for $\alpha =\f{1-\epsilon_0}{2}$ and let
\[
\epsilon_j^{N} = \f{1}{N} \sum_{i \neq j}\beta_N \cdot \Lambda \cdot G(y_j^N,y_i^N,k) b(y_i^N) -
\int_{\Omega} G(y_j^N, y,k) V(y) b(y) dy.
\]
Then there exists a function $\tilde{\epsilon} \in X$ such that
$\tilde{ \epsilon}(y_j^N) = \epsilon_j^{N}$ and 
$$
 \|\tilde{\epsilon}\|_X \lesssim N^{-\f{1-\epsilon_0}{6}}\|b\|_X.  
$$
\end{lem}

\begin{proof} 

We only give a proof for the case when the wave number $k=0$. For the case when $k \neq 0$, we can first decompose $G(x, y, k)$ into the singular part, 
$G(x, y, 0)$,  and a smooth part and then decompose $\epsilon_j^{N}$ accordingly. The singular part corresponds exactly to the case when $k=0$, while the smooth part can be handled in a straightforward way. 

Step 1. Let 
\[ g_N(r) =
 \begin{cases}
    -\f{1}{4 \pi r},      & \quad \text{if } r \geq r_N, \\
    \f{g(r_N)}{r_N} r, & \quad \text{if } 0\leq r < r_N,\\
  \end{cases}
  \]
and $$G_N(x, y) = g_N(|x-y|).$$ 
Define 
\[
\tilde{\epsilon}(y) = \f{1}{N} \sum_{1\leq i \leq N}\beta_N \cdot \Lambda \cdot G_N(y,y_i^N) b(y_i^N) -
\int_{\Omega} G(y, x, 0) V(x) b(x) dx. 
\]
It is clear that $\tilde{\epsilon}(y_j^N) = \epsilon_j^N$.  
 
Step 2.  We show that 
\be \label{estimate-step2}
\max_{x\in \Omega } |\tilde{\epsilon}(x)| \lesssim \f{1}{N^{\f{1-\epsilon_0}{3}}} \|b\|_X.
\ee

Indeed, for each $y \in \Omega$, let $y_l^N \in \{ y_j^N\}_{1\leq j \leq N}$ be such that
\[
| y- y_l^N| =\min_{1\leq j \leq N} |x- y_j^N|.
\]
It is clear that $| y- y_l^N| =O(r_N) = O(\f{1}{N^{\f{1}{3}}})$. By Lemma \ref{lem-technique2}, \beas   
\f{1}{N} \sum_{1\leq i \leq N}\beta_N \cdot \Lambda \cdot G_N(y,y_i^N) b(y_i^N) = \f{1}{N}  \sum_{1\leq i \leq N}\beta_N \cdot \Lambda \cdot G_N(y_l^N,y_i^N) b(y_i^N)
+ O(\f{1}{N^{\f{1}{3}}}). 
\eeas
On the other hand, by Lemma \ref{lem-limit-operator},
$$
| \int_{\Omega} G(y, x, 0) V(x) b(x) dx - \int_{\Omega} G(y_l^N, x, 0) V(x) b(x)  dx| 
\lesssim |y-y_l^N|^{1-\epsilon_0} \cdot \|b\|_X \lesssim \f{1}{N^{\f{1-\epsilon_0}{3}}}\|b\|_X. 
$$

Therefore, we can conclude that
$$
|\tilde{\epsilon}(y)- \tilde{\epsilon}(y_l^N) | \lesssim 
\f{1}{N^{\f{1-\epsilon_0}{3}}}\|b\|_X. 
$$ 
This together with
$$
\max_{1\leq j \leq N } | \tilde{\epsilon}(y_j^N)| \lesssim N^{-\alpha} \|b\|_X,
$$
which follows from Lemma \ref{lem-asump}, yields the desired estimate 
(\ref{estimate-step2}).

Step 3. As a consequence of estimate (\ref{estimate-step2}), we have
\be \label{estimate-step3}
\max_{x, x+h \in \Omega } |\tilde{\epsilon}(x+h) -\tilde{\epsilon}(x)| \lesssim \f{1}{N^{\f{1-\epsilon_0}{3}}} \|b\|_X.
\ee

Step 4. We show that
\be  \label{estimate-step4}
\max_{x, x+h \in \Omega } |\tilde{\epsilon}(x+h) -\tilde{\epsilon}(x) | \lesssim  |h|^{1-\epsilon_0} \cdot \|b\|_X.
\ee
In fact, by Lemma \ref{lem-limit-operator}, we have 
\[
\max_{x, x+h \in \Omega } | \int_{\Omega} G(x+h, y, 0) V(y) b(y) dy - \int_{\Omega} G(x, y, 0) V(y) b(y) dy | \lesssim | h|^{1-\epsilon_0} \cdot \|b\|_X.
\]
This together with Lemma \ref{lem-technique2} proves estimate (\ref{estimate-step4}).

Step 5. Finally, combining (\ref{estimate-step3}) and (\ref{estimate-step4}), we get 
\[
\max_{x, x+h \in \Omega } |\tilde{\epsilon}(x+h) -\tilde{\epsilon}(x)| \lesssim \f{1}{N^{\f{1-\epsilon_0}{6}}} \cdot h^{\f{1-\epsilon_0}{2}} \cdot \|b\|_X.
\]
Therefore, we have shown that 
$\tilde{e} \in X=C^{\alpha}(\Omega) $ with $\alpha= \f{1-\epsilon_0}{2}$ and $\|\tilde{e}\|_X  \lesssim \f{1}{N^{\f{1-\epsilon_0}{6}}} \|b\|_X$. 
Hence, the lemma is proved. 

\end{proof}

Finally, we present a technical lemma which is used in the proof of Lemma \ref{lem-epsilon}.

\begin{lem}  \label{lem-technique2}
Let $G_N(\cdot, \cdot)$ be defined as in Step 1 of the proof of Lemma \ref{lem-epsilon}. Then following estimate holds
\be \label{inequality2}
|\f{1}{N} \sum_{1\leq i \leq N}[G_N(x+h,y_i^N) - G_N(x,y_i^N)]b(y_i^N) | \lesssim |h|^{1-\epsilon_0} \cdot \|b\|_{C^0(\overline{\Omega})},
\ee
for all $x, x+h \in \Omega$.
\end{lem}

\begin{proof}
Denote 
\[
\epsilon^N_1(x)= \f{1}{N} \sum_{1\leq i \leq N}(G_N(x+h,y_i^N) - G_N(x,y_i^N))b(y_i^N).
\]
There are two cases: 
Case 1. $|h| \leq 2 r_n$; Case 2. $|h| \geq 2 r_n$.
We first show (\ref{inequality2}) in Case 1. 

We introduce the following two sets of indices. 
\beas
J_1^N(x, h)&=& \{j: |x-y_j^N| < 2 r_N  \,\, \mbox{or }\,\,|x+h-y_j^N| < 2 r_N \},\\
J_2^N(x, h)&=& \{j: 1\leq j \leq N, j \notin J_1^N(x, h)\}.
\eeas
Then,
\beas
\epsilon^N_1(x)&=& 
 \f{1}{N} \sum_{i \in J_1^N(x,h)}(G_N(x+h,y_i^N) - G_N(x,y_i^N))b(y_i^N)\\
 && +\f{1}{N} \sum_{i \in J_2^N(x,h)}(G_N(x+h,y_i^N) - G_N(x,y_i^N))b(y_i^N) \\
&:=& \epsilon^N_{1,1}(x)+\epsilon^N_{1,2}(x).  
\eeas
By Assumption \ref{assump0}, there exists a constant $C_0$, independent of $N$ such that 
the number of elements in $J_1^N(x, h)$ is bounded by $C_0$. 
On the other hand, since $|g'_N(r)| \lesssim \f{1}{r_N^2}$ for all $r$, we can show that
\[
|G_N(x+h,y_i^N) - G_N(x,y_i^N)| \lesssim \f{1}{r_N^2}|h|. 
\]
Therefore,
\be \label{inequality2-1}
|\epsilon^N_{1,1}(x)| \lesssim \f{1}{Nr_N^2}|h| \cdot \|b\|_{C^0(\overline{\Omega})} \lesssim \f{1}{N^{\f{1}{3}}}\cdot |h| \cdot\|b\|_{C^0(\overline{\Omega})}. 
\ee

Next, for each $j \in J_2^N(x, h)$, note that $G_N(x,y_j^N)= \f{1}{|x-y_j^N|}$,
$G_N(x+h,y_j^N)= \f{1}{|x+h-y_j^N|}$. 
Thus,
\beas
\epsilon^N_{1,2}(x)&=&
|\f{1}{N} \sum_{ j \in J_2^N(x, h)}(\f{1}{|x-y_j^N|} - \f{1}{|x+h-y_j^N|})b(y_j^N) |\\
& \leq &
\f{1}{N} \sum_{ j \in J_2^N(x, h)}
\f{|h|}{|x-y_j^N|\cdot |x+h-y_j^N|} \cdot
 \|b\|_{C^0(\overline{\Omega})} \\
& \leq & \f{\|b\|_{C^0(\overline{\Omega})} \cdot |h|}{N} \left(  \sum_{ j \in J_2^N(x, h)}\f{1}{|x-y_j^N|^2}
+ \sum_{ j \in J_2^N(x, h)}\f{1}{|x+h-y_j^N|^2} \right), 
\eeas
which further yields  
\be  \label{inequality2-2}
|\epsilon^N_{1,2}(x)| 
\lesssim  \|b\|_{C^0(\overline{\Omega})} |h|^{1-\epsilon_0}, 
\ee 
by Assumption \ref{assump2}. 
Combining (\ref{inequality2-1})-(\ref{inequality2-2}), we obtain (\ref{inequality2}). 

Now, we show (\ref{inequality2}) in Case 2. Denote 
\beas
J_1^N(x, h)&=& \{j: |x-y_j^N| \leq  2 r_N  \,\, \mbox{or }\,\,|x+h-y_j^N| \leq 2 r_N \};\\
J_2^N(x, h)&=& \{j: |x-y_j^N| > 2 r_N, |x+h-y_j^N| > 2 r_N,  |x-y_j^N| \leq h, \,\,\mbox{or}\,\, |x+h-y_j^N| \leq h \};\\
J_3^N(x, h)&=& \{j: |x-y_j^N| >h,|x+h-y_j^N| > h \}.
\eeas

Then,
\beas
\epsilon^N_1(x)
&=& \f{1}{N} \sum_{i \in J_1^N(x,h)}(G_N(x+h,y_i^N) - G_N(x,y_i^N))b(y_i^N)\\
&& + \f{1}{N} \sum_{i \in J_2^N(x,h)}(G_N(x+h,y_i^N) - G_N(x,y_i^N))b(y_i^N) \\
&& + \f{1}{N} \sum_{i \in J_3^N(x,h)}(G_N(x+h,y_i^N) - G_N(x,y_i^N))b(y_i^N) \\
&:=& \epsilon^N_{1,1}(x)+\epsilon^N_{1,2}(x) +\epsilon^N_{1,3}(x).
\eeas
Following the same argument as in Case 1, we can show that
\be  \label{eps2-1}
|\epsilon^N_{1,1}(x) | \lesssim \f{1}{N^{\f{1}{3}}}|h| \cdot\|b\|_{C^0(\overline{\Omega})},
\ee
and 
\be  \label{eps2-2}
|\epsilon^N_{1,3}(x) | \lesssim \|b\|_{C^0(\overline{\Omega})} \cdot |h|^{1-\epsilon_0}. 
\ee 
We now consider $\epsilon^N_{1,2}(x)$. We have
\beas
 |\epsilon^N_{1,2}(x) | & \leq &
\f{1}{N} \sum_{ j \in J_2^N(x, h)}
\f{|h|}{|x-y_j^N|\cdot |x+h-y_j^N|}
 \|b\|_{C^0(\overline{\Omega})} \\
& \leq & \f{\|b\|_{C^0(\overline{\Omega})} }{N} \left( \sum_{\substack{
   j \in J_2^N(x, h) \\
   |x-y_j^N| \geq \f{h}{2}
  }} \f{|h|}{|x-y_j^N|\cdot |x+h-y_j^N|}
+ \sum_{\substack{
   j \in J_2^N(x, h) \\
   |x+h-y_j^N| \geq \f{h}{2}
  }} \f{|h|}{|x-y_j^N|\cdot |x+h-y_j^N|} \right)  \\
  & \leq & \f{2\|b\|_{C^0(\overline{\Omega})} }{N} \left( \sum_{\substack{
   j \in J_2^N(x, h) \\
   |x-y_j^N| \geq \f{h}{2}
  }} \f{1}{|x+h-y_j^N|}
+ \sum_{\substack{
   j \in J_2^N(x, h) \\
   |x+h-y_j^N| \geq \f{h}{2}
  }} \f{1}{|x-y_j^N|} \right) \\
  & \leq & \f{2\|b\|_{C^0(\overline{\Omega})} }{N} \left( \sum_{\substack{
   |x+h-y_j^N| \leq 3h \\
   |x+h-y_j^N| \geq 2r_N
  }} \f{1}{|x+h-y_j^N|}
+ \sum_{\substack{
    |x-y_j^N| \leq 3h \\
   |x-y_j^N| \geq 2r_N
  }} \f{1}{|x-y_j^N|} \right) 
   \lesssim   \|b\|_{C^0(\overline{\Omega})} \cdot |h|, 
\eeas
where we have used Assumption \ref{assump2}
in the last inequality. 
This combined with (\ref{eps2-1})-(\ref{eps2-2}) proves (\ref{inequality2}) in Case 2.
The proof of the lemma is then complete. 
\end{proof}




\end{document}